\documentclass[11pt, preprint]{elsarticle}

\usepackage{fullpage}
\usepackage[cp1251]{inputenc}
\usepackage[T2A]{fontenc}       
\usepackage{caption}
\usepackage{graphicx}
\usepackage{amsmath}
\usepackage{amssymb}
\usepackage{amsthm}
\usepackage{color}

\allowdisplaybreaks[4]

\DeclareMathOperator{\Mob}{Mob}

\newlength{\ps}

\newcommand{\Skip}[1]{}

\newtheorem{Prop}{Proposition}

\newsavebox{\TmpBox}
\newcommand{\tmp}{}
\newlength{\tmplength}                    %

\newcommand{\Equa}[2]{
\begin{equation}#2\label{#1}\end{equation}%
}
\newcommand{\equa}[1]{\[ #1 \]}

\newcommand{\ceq}{\mathrel{\raisebox{0.5pt}{:}\mspace{-3mu}=}}

\newcommand{\St}[2][]{{#2}_{#1}^{\star}}
\newcommand{\abs}[1]{\left\lvert#1\right\rvert}

\newcommand{\sgn}{\mathop{\rm sgn}\nolimits}
\newcommand{\Deg}[1]{{\ifmmode{#1}^\circ\else${#1}^\circ$\fi}}
\newcommand{\e}{{\mathrm{e}}}
\newcommand{\Exp}[1]{\e^{#1}}
\newcommand{\nvec}[1]{\mathbf{n}(#1)}
\newcommand{\iu}{{\mathrm{i}}}
\newcommand{\ber}[2]{{\cal B}^{(#1)}_{#2}}
\newcommand{\Ber}[3]{\ber{#1}{#2}\left(#3\right)}

\renewcommand{\le}{\leqslant}

\renewcommand{\ge}{\geqslant}
\newcommand{\ige}{\,{\ge}\,}
\newcommand{\Keep}[3]{\left\{#1,\, #2,\, #3 \right\}}

\newcommand{\Eqref}[1]{\buildrel\eqref{#1}\over=}
\newcommand{\GT}[1]{$#1>0$}
\newcommand{\GE}[1]{$#1\ge0$}
\newcommand{\LT}[1]{$#1<0$}
\newcommand{\LE}[1]{$#1\le0$}
\newcommand{\EQ}[1]{$#1=0$}
\newcommand{\NE}[1]{$#1\not=0$}
\newcommand{\So}{\quad\Longrightarrow\quad}

\newcommand{\HM}{\hphantom{{-}}}

\newcommand{\Kl}[1]{{\ifmmode{\cal K}_{#1}\else${\cal K}_{#1}$\fi}}

\def\FigDir{}
\def\figurename{Fig.}
\newcommand{\Figref}[1]{\ref{fig:#1}}
\newcommand{\RefFig}[2][]{\figurename\,\Figref{#2}{#1}} 
\newcommand{\Infigw}[2]{
\includegraphics[width=#1]{\FigDir#2.eps}}

\newcommand{\Bfig}[3]{
\setlength{\tmplength}{-#1}\addtolength{\tmplength}{\textwidth}%
\parbox[b]{\textwidth}{%
\makebox[.5\tmplength][l]{}%
\includegraphics[width=#1]{\FigDir#2.eps}%
\makebox[.5\tmplength][r]{}%
}%
\caption{{\small #3}}\label{fig:#2}%
}

\begin{document}

\author{A.\,I.~Kurnosenko}

\address{Private research (Russian Federation)}

\title
{Two-point G$^2$ Hermite interpolation with spirals by inversion of conics: summary}

\begin{abstract}
The article completes the research of two-point G$^2$ Hermite interpolation
problem with spirals by inversion of conics.
A simple algorithm is proposed to construct a family of 4th degree rational spirals,
matching given G$^2$ Hermite data.
A possibility to reduce the degree to cubic is discussed.
\end{abstract}

\maketitle{}

\section{Introduction}

This note is intended to complete our research \cite{AKparab,AKhyperb} 
of the problem of two-point G$^2$ Hermite interpolation with spirals by inversion of conics.
The review of the problem was given in~\cite{AKparab},
together with explanation of the general idea of applying inversion to construct a spiral interpolant.
M\"oebius maps of a parabolic arc have been considered.
In \cite{AKhyperb} the theory was augmented by including {\em long spirals}. 
The construction was based on another special kind of conic,
namely, a hyperbola with parallel tangents at the endpoints.

It is now clear that, for a given two-point G$^2$ data,
there exists a family of solutions, produced by involving other conics.
The questions naturally arise: could we propose to a designer a possibility to select
a curve among a family of shapes and curvature profiles?
Is there, among the family of rational quartic spirals, a curve, 
reducible to cubic? 
Even if the answers are not much positive, these questions should be answered.

The rest of the article is organized as follows.
In Sections \ref{sec:Extend} and \ref{sec:SpiralConic}
we consider conics with non-positive weights in
the standard rational form of a conic, and explore them for spirality.
Constructing the family of interpolants is described in Section~\ref{sec:Construction};
in particular, subsections \ref{sec:Given} and \ref{sec:StepByStep}
are supposed to be sufficient to design the corresponding script,
omitting theoretical details. 
Figures \Figref{Fig1-Exa}, \Figref{Fig2-LogSp}, \Figref{Fig4-Cum}, \Figref{Fig8-Cornu}
illustrate the families under discussion.
Finding rational cubic spiral is considered in Section~\ref{sec:Cubic}.

\begin{figure}[t]
\centering%
\Bfig{.88\textwidth}{Fig1-Exa}{%
Given G$^2$ data 
$\St[A]{{\cal K}} = \{-1,0,\Deg{-150},-0.4\}$ and 
$\St[B]{{\cal K}} = \{1,0,\Deg{-120},0.3\}$
(circles of curvature are shown dashed),
$\St\sigma=\Deg{90}$.
A family of spiral interpolants and their curvature plots ($s$~is arc length).
Curves are identified by the family parameter $\theta$ (in degrees).}
\end{figure}

\begin{figure}[t]
\centering%
\Bfig{.72\textwidth}{Fig2-LogSp}{%
Given G$^2$ data is borrowed from a logarithmic spiral, represented by dots.
A family of spiral interpolants $AB$ and their curvature plots.
}
\end{figure}

\section{Extention of rational quadratic B\'ezier representation of a conic}\label{sec:Extend}

Using 2nd degree rational curves in CAD applications
was restricted to continuous ones. Discontinuities, possibly occurring
in hyperbolas, were avoided. 
This type of conics proved useful to construct spirals.
To include it into the standard rational quadratic form of conic~\cite{Farin},
\Equa{Ratw0w1w2}{%
    \vec{r}(t)=\dfrac{P_0w_0(1{-}t)^2+2P_1w_1(1{-}t)t+P_2w_2t^2}%
                       {w_0(1{-}t)^2+2w_1(1{-}t)t+w_2t^2},\quad w_0=1,\quad w_2\not=0,
}
\Skip{%
Taking $k=-\frac{1}{\sqrt{\abs{w_2}}}-1$
the rest of the curve, corresponding to 
$0\ige t\ige{-\infty}$ and
${+\infty}\ige t\ige 1$, is reparametrized to $t\in[0,1]$. 
}
we assume non-positive weights $w_{1,2}$.
Linear rational reparametrization
\equa{%
    t \to \frac{t}{t+(1-t)\sqrt{\abs{w_2}}}
} 
maps the segment $t\in[0;1]$
onto itself continuously, and replaces weights $\{1,w_1,w_2\}$ by\vspace{-.5\baselineskip}
\Equa{w1w2w3}{%
    \left\{1,\pm\dfrac{w_1 }{\sqrt{\abs{w_2}}},\sgn w_2 \right\}=\{1,w,j\}, \quad j=\pm 1.
}
Conics with parallel end tangents \cite[Sec.\,12.8]{Farin} can be included into consideration
by assuming~$w$ tending to zero, while the control point $P_1=(p,q)$ tends to infinity:
\Equa{PQinf}{%
    w\to 0, \quad \sqrt{p^2+q^2}\to\infty,\quad\text{products}\quad p_w= pw \text{~~and~~} q_w=qw\not=0
}
remaining finite.
With weights \eqref{w1w2w3}, and the {\em normalized position} of an arc \cite{AKparab},
namely, 
\equa{
       P_0=A=(-1,0),\qquad P_2=B=(1,0),
}       
Eq.~\eqref{Ratw0w1w2}
looks like
\Equa{RatConic}{%
  \begin{array}{l}
    x(t)=\dfrac{X(t)}{W(t)},\\[1ex]y(t)=\dfrac{Y(t)}{W(t)},
  \end{array}
  \quad\mbox{where}\quad\left\{%
  \begin{array}{l}
      X(t)=-(1{-}t)^2+2p_w(1{-}t)t+jt^2,\\[1ex]
      Y(t)=\HM{}2q_wt(1-t),\\[1ex]
      W(t)=\HM{}(1{-}t)^2+2w(1{-}t)t+jt^2.
  \end{array}\right.
}
The sides of the control polygon, $h_1=\abs{AP_1}$ and $h_2=\abs{P_1B}$,
also have weighted versions $h_{1_w}=\abs{wh_1}$ and $h_{2_w}=\abs{wh_2}$,
finite in the case of infinite control point~\eqref{PQinf}:
\Equa{h1h2}{
  \begin{array}{ll}
     h_1=\abs{AP_1}=\sqrt{(1+p)^2+q^2},&h_2=\abs{P_1B}=\sqrt{(1-p)^2+q^2},\\
     h_{1_w}=\sqrt{(w+p_w)^2+q_w^2},&h_{2_w}=\sqrt{(w-p_w)^2+q_w^2}.
  \end{array}
}

\Skip{%
\equa{%
  \begin{array}{ll}
     x'=\dfrac{X'W-XW'}{W^2},\quad&x''=\dfrac{X''W-XW''}{W^4}-\dfrac{2(X'W-XW')W'}{W^3};\\
     y'=\dfrac{Y'W-YW'}{W^2},\quad&y''=\dfrac{Y''W-YW''}{W^4}-\dfrac{2(Y'W-YW')W'}{W^3}.
  \end{array}
}  
\equa{
      g(t)=\sqrt{x'^2+y'^2}=\dfrac{\sqrt{G(t)}}{W^2(t)},\quad  
      f(t)=y''x'-x''y'= -8j\dfrac{q_w}{W^3(t)},
}
where
\equa{%
  \begin{array}{l}
      G(t)=(X'W-XW')^2+(Y'W-YW')^2,\\
      G'(t)=2(X'W-XW')(X''W-XW'') + 2(Y'W-YW')(Y''W-YW'').
  \end{array}
}}  
\Skip{%
\Equa{ktdkt}{%
      k(t)=\dfrac{f(t)}{g^3(t)}=-8jq_w\dfrac{W^3}{G^{3/2}},\qquad
      k'_t=-12jq_w W^2\dfrac{2W'G-WG'}{G^{5/2}}.
}
\equa{%
   \begin{array}{lll}
              t=0: &&\\    
    X= -1,&Y=0,&W=1;\\       
    X'= 2(1{+}P),&Y'=2Q,&W'=2(w{-}1);\\       
    X''= 2(w_2{-}1{-}2P),&Y''=-4Q,&W=2(w_2{+}1{-}2w);\\  
    G=  ,& G'=; &\\
    \strut&&\\
    t=1:&&\\
    X= w_2,&Y=0,&W=w_2;\\
    X'= 2(w_2{-}P),&Y'=-2Q,&W'=2(w_2{-}w);\\
    X''= 2(w_2{-}1{-}2P),&Y''=-4Q,&W=2(w_2{+}1{-}2w);\\
    G=  ,& G'=; &    
   \end{array}
}}
Calculating corresponding derivatives yields the direction $\tau(t)$ of the tangent vector
$(\cos\tau(t),\sin\tau(t))^T$,
and the curvature function
\Equa{ktdkt}{%
      k(t)=-8jq_w\dfrac{W(t)^3}{G(t)^{3/2}},\quad\text{where}\quad
      G(t)=[X'(t)W(t)-X(t)W'(t)]^2+[Y'(t)W(t)-Y(t)W'(t)]^2.
}
Boundary G$^2$ data 
$\{-1,\,0,\,\alpha{=}\tau(0),\,a{=}k(0)\}$ and
$\{1,\,0,\,\beta{=}\tau(1),\,b{=}k(1)\}$
for conic arc~\eqref{RatConic} are:
\Equa{ConicG2}{%
\hspace*{-5mm}%
   \begin{array}{lll}
    \cos\alpha = \dfrac{w}{\abs{w}}\cdot\dfrac{1+p}{h_1}= \dfrac{w+p_w}{h_{1_w}},\;&
    \sin\alpha = \dfrac{w}{\abs{w}}\cdot\dfrac{q}{h_1}= \dfrac{q_w}{h_{1_w}},&
    a=-j\dfrac{w}{\abs{w}}\cdot \dfrac{q}{w^2 h_1^3}=-j\dfrac{q_w}{h_{1_w}^3};\\[1.5ex]
    \cos\beta = j \dfrac{w}{\abs{w}}\cdot\dfrac{1-p}{h_2}= j\dfrac{w-p_w}{h_{2_w}},\;&
    \sin\beta = -j\dfrac{w}{\abs{w}}\cdot\dfrac{q}{h_2}= -j\dfrac{q_w}{h_{2_w}},\quad&
    b=-\dfrac{w}{\abs{w}}\cdot\dfrac{q}{w^2 h_2^3}=-\dfrac{q_w}{h_{2_w}^3}.
   \end{array}
}
Fraction $\frac{w}{\abs{w}}$ for $w=0$ is considered as the corresponding limit, equal to~$\pm1$,
according to particular application.

\section{Spiral conic arcs}\label{sec:SpiralConic}

The
implicit equation $a_{11}x^2+2a_{12}xy+a_{22}y^2+2a_{10}x+2a_{20}y+a_{00}=0$
of curve \eqref{RatConic} is
\Equa{Implicit}{%
     q_w^2x^2 - 2p_wq_w xy + (p_w^2+j-w^2)y^2+2wq_wy-q_w^2=0.
}
Invariants of the quadratic form are
\equa{
     I_3=\left|
       \begin{array}{ccc}
          a_{11} & a_{12} & a_{10}\\
          a_{12} & a_{22} & a_{20}\\
          a_{10} & a_{20} & a_{00}
       \end{array}
       \right| = -j q_w^4,\qquad
   \begin{array}{l}
      I_2=
     \left|
      \begin{array}{cc}
         a_{11} & a_{12} \\
         a_{12} & a_{22} 
      \end{array}
     \right| =  (j-w^2)q_w^2,\\[2ex]
       I_1 = a_{11}+a_{22}=p_w^2+q_w^2+j-w^2.
    \end{array}
}
This is conic when \NE{I_3}, i.e. \NE{q_w}.\medskip

\noindent%
1. Let $j=1$. Then $I_3<0$, \ $I_2= q_w^2 (1{-}w^2)$, \ $I_1 = p_w^2+q_w^2+1-w^2$, i.\,e.
\equa{
   \begin{array}{llll}
          |w|<1:& \text{ellipse}& (I_1 I_3<0, \quad &I_2>0);\\
          |w|=1:& \text{parabola}& (I_3\neq 0,\quad  &I_2=0);\\ 
          |w|>1:& \text{hyperbola}& (I_3\neq 0,\quad &I_2<0).
   \end{array}
}
Conics with $j=1, w>0$ have been studied for spirality in \cite{FreyConic}.
Let us exclude non-spiral cases $w\le0$.
\begin{itemize}\setlength{\parindent}{0pt}
\item[a)]
Ellipse with \EQ{w} has (anti)parallel end tangents, is centered in the origin,
and the arc $AB$ includes one or two vertices between endpoints.
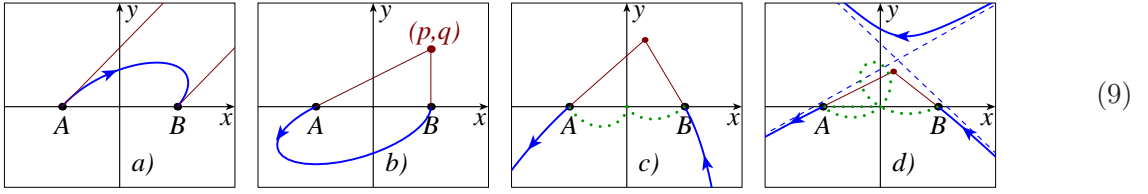
\Equa{NonSpirals}{%
   \parbox[c]{.9\linewidth}{\Infigw{.94\linewidth}{Pic09}}
}
\item[b)]
For an ellipse with \ $-1<w<0$ \ we see from \eqref{ConicG2},
that either $\abs{\alpha}>\pi/2$
($\cos\alpha<0$ when $p>{-1}$),
or $\abs{\beta}>\pi/2$ (when $p<1$), or both.
Such elliptic arc definitely includes vertices.
\item[c)]
Parabola with $w=-1$ includes the infinite point at $t=\frac12$, producing a cusp under inversion
(dotted curve).
\item[d)]
Hyperbolic infinite points are acceptable.
But both roots
$t_{1,2}=\frac12\left(1\pm\sqrt{\frac{w+1}{w-1}}\right)$
of the equation $W(t)=0$ are inside the interval $(0;1)$. The curve includes the entire branch
$t\in(t_1;t_2)$ with its vertex.
\end{itemize}

\noindent%
2. Let  $j=-1$. Then $I_3\neq 0$, \ $I_2<0$.   
This is a hyperbola 
with discontinuities at
\Equa{DiscHyp}{%
    \begin{array}{l}
        t_1=\dfrac{w{-}1 - \sqrt{1+w^2}}{2w},\quad
        t_2=\dfrac{w{-}1+\sqrt{1+w^2}}{2w}=\dfrac1{1{-}w+\sqrt{1+w^2}};\\
      \aligned   
        w>0:&&\quad  t_1<0<&t_2<1;\\
        w<0:&&\quad  0<&t_2<1<t_1;\\
        w=0:&&\quad    &t_2=\tfrac12,\quad t_1=\infty.
      \endaligned 
    \end{array}
}
Exactly one discontinuity,~$t_2$, falls into the interval $t\in(0;1)$.
The hyperbola may have no vertices.\medskip

The above considerations have excluded some evident non-spiral cases.
The rest requires more detailed analysis for the absence of vertices.

Let \EQ{K_{1,2}(x,y)} be equations of two circles,
\Equa{K12xy}{%
   \begin{array}{l}
       K_1(x,y)=(x+x_w)^2+y^2 -r_w^2,\\[1ex]
       K_2(x,y)=(x-x_w)^2+y^2 -r_w^2,
   \end{array}
\qquad r_w=\dfrac{j}{2w^2},\quad x_w=1{-}r_w,
}
centered at $(\mp x_w,0)$, both of radius $\abs{r_w}$.
The first one (the left one) passes through point $A=(-1,0)$, the second through $B=(1,0)$.
If $r_w={+1}$, two circles are coincident with the unit circle.
\begin{Prop}
\begin{subequations}\label{LimCirc}
Convex conic $(j=1)$ is vertex-free if
\Equa{LimCirc1}{%
     w\ge\frac{1}{\sqrt{2}},\quad p\not=0,\quad K_1(p,q)\cdot K_2(p,q) \le 0.
}     
Discontinuous conic $(j=-1)$ is vertex-free if
\Equa{LimCirc2}{%
    \begin{array}{lll}
       \text{either}&w>0, & K_1(p,q)\le 0,\\
       \text{or} &w<0,    & K_2(p,q)\le 0.
    \end{array}
}
\end{subequations}
\end{Prop}

\begin{figure}[t]
\centering%
\Bfig{1.\textwidth}{Fig3-K1K2}{Bounding circles \eqref{K12xy}}%
\end{figure}

Four examples of convex cases are shown at the left side of \RefFig{Fig3-K1K2}.
Condition $w\ge\frac{1}{\sqrt{2}}$ is equivalent to $0<r_w\le 1$.
The proof of~\eqref{LimCirc1} is given in \cite{FreyConic} as Theorem\,9.1.

\begin{proof}[{\bf Proof}]
To prove \eqref{LimCirc2}, we first link expressions for $K_{1,2}(p,q)$,
\equa{
  K_1(p,q)=\left[p+\left(1{-}r_w\right)\right]^2+q^2 -r_w^2
          =\frac{w^2h_1^2-j(1{+}p)}{w^2},\quad
  K_2(p,q)=\ldots=\frac{w^2h_2^2-j(1{-}p)}{w^2},        
}
with derivatives of curvature \eqref{ktdkt}:
\equa{
     k'(0)=-6jq\dfrac{jw^2 h_1^2-j(p+1)}{\abs{w}^3h_1^5}
     \equiv -6jq\dfrac{K_1(p,q)}{\abs{w} h_1^5},  \qquad
     k'(1)= 6jq\dfrac{jw^2 h_2^2-j(1-p)}{\abs{w}^3h_2^5}
     \equiv 6jq\dfrac{K_2(p,q)}{\abs{w}h_2^5}.
}

Let $w>0$, $q>0$.
There is exactly one discontinuity \eqref{DiscHyp} at $t=t_2\in(0,1)$.
The curve consists of two infinite branches, $t\in[0;t_2)$ and $t\in(t_2;1]$.
Positive curvature $k(0)=a>0$ \eqref{ConicG2} approaches zero as the curve approaches
the asymptota ($t\to t_2$),\\
\indent
(a)~either increasing $[k'(0)>0]$ up to the vertex, and then decreasing to zero;\\%
\indent
(b)~or monotonously decreasing to zero $[k'(0)\le 0]$.\\%
The case (b) means the absence of vertex in $t\in[0,t_2]$.
Decreasing $[k'(t)\le 0]$ must continue in $t\in(t_2;1]$ up to $t=1$,
not turning into increasing:
\equa{%
   k'(0)\le 0 \Longleftrightarrow K_1(p,q)\le 0,\qquad
   k'(1)\le 0 \Longleftrightarrow K_2(p,q)\ge 0.
}
Provided $K_1\le0$, the condition $K_2>0$
holds automatically:
any point inside the left circle is outside the right one. 

With \GT{w}, but \LT{q}, negative curvature \LT{a=k(0)} must increase to zero when $t\to t_2$,
and must continue increasing in $(t_2;\,1]$: 
\GE{k'(0)} and \GE{k'(1)} also yield \LE{K_1(p,q)} and \GE{K_2(p,q)}.

For \LT{w} we obtain similarly \LE{K_2(p,q)}, and automatical \GT{K_1(p,q)}.
\end{proof}

To distinguish the cases of increasing/decreasing curvature,
consider the difference $b-a$ of end curvatures, whose sign, under condition of curvature monotonicity,
is M\"oebius invariant:
\equa{%
  \aligned
    &      
     \makebox[0mm][l]{$b-a= -\dfrac{w}{\abs{w}}\dfrac{q}{w^2 h_1^3 h_2^3}\left(h_1^3-{j}h_2^3\right);$}\\
    &j=+1\;(w>0){:}       && \sgn(h_1^3{-}h_2^3)=\sgn p  &\Longrightarrow\quad\sgn(&b-a)=-\sgn(pq);\\
    &j=-1\;(w\lessgtr0){:}&& \sgn(h_1^3{+}h_2^3)=1       &\Longrightarrow\quad\sgn(&b-a)=-\sgn(qw).
 \endaligned
}
The unit circle $p^2+q^2\le1$ covers all smaller circles~\eqref{LimCirc1}.
Two halfplanes $\abs{p}>1$ cover all circles~\eqref{LimCirc2}.
Two shaded sectors (including circular boundaries), and two shaded quadrants
cut therefrom the regions,
where a control point $(p,q)$ could be located, 
generating a spiral conic arc with {\em increasing} curvature, $\sgn(b-a)=1$:
\Equa{Wsign}{
   b>a:\quad\parbox[c]{200pt}{\Infigw{200pt}{Pic13}}
}

\section{Construction of the family of spirals}\label{sec:Construction}

\subsection{Given data}\label{sec:Given}

The set of G$^2$ data, denoted as ${\cal K}= \{x,\,y,\,\tau,\,k\}$,
includes point $(x,y)$, direction of unit tangent $\nvec{\tau}=(\cos\tau,\,\sin\tau)$,
and curvature~$k$ at this point.
Given
G$^2$ data is assumed to be {\em normalized}:
\Equa{GivenKK}{%
   \St[A]{{\cal K}} = \{-1,0,\St\alpha,\St a\},\qquad \St[B]{{\cal K}} = \{1,0,\St\beta,\St b\}.
}
Superscript $\St{}$ is used to denote given data or derived quantities, such as
\Equa{DefQ}{%
  \St[1]{g}=\St a+\sin\St\alpha,\quad \St[2]g=\St b-\sin\St\beta,\qquad 
  \St Q=\St[1]{g}\St[2]{g}+\sin^2\dfrac{\St\alpha+\St\beta}{2}<0.
}
Condition \LT{\St Q} means that a non-biarc spiral, matching given data, exists
(see \cite[St.\,2]{AKparab}).

In the following sections only the case of increasing curvature
is considered.
For decreasing curvature it is proposed to apply symmetry about the $x$-axis
by replacing
\Equa{MakeIncr}{
   \St\alpha\ceq{-\St\alpha},\quad
   \St\beta\ceq{-\St\beta},\quad
   \St a \ceq {-\St a},\quad
   \St b \ceq {-\St b}\quad
   \left[\St[1]g \ceq {-\St[1]g},\quad
   \St[2]g \ceq {-\St[2]g}\right].
}
For increasing curvature $\St[1]g<0$, $\St[2]g>0$ (\cite[St.\,6]{AKparab}).

After bringing boundary angles $\St\alpha,\St\beta$ to the range $(-\pi;\pi]$
the condition $\St\alpha+\St\beta>0$ should be verified 
(Vogt's theorem, see \cite[Sec.\,2]{AKhyperb}).
If it holds, {\em a short spiral} exists, matching given data~\eqref{GivenKK}.
Otherwise a spiral is forced to make a turn near one of the endpoints,
thus becoming {\em long}. 
Continuous ({\em cumulative}) definition of boundary angles (see, e.\,g., Sec.\,3.3 in \cite{Bilens})
requires correction,
either $\St\alpha\ceq\St\alpha+2\pi$, or $\St\beta\ceq\St\beta+2\pi$.
%
%
We do not need to know, which one should be applied, and do not apply it to $\St\alpha,\St\beta$.
It is sufficient to correct the value of the invariant $\St\sigma$,
and to know exact values of sines and cosines of half-angles 
$\frac{\St\alpha\pm\St\beta}2$.
Define therefore\vspace{-\baselineskip}
\Equa{DefSigma}{%
  \aligned
   &\text{if~~}\St\alpha+\St\beta>0: &&\St\sigma = \St\alpha+\St\beta,
                                     &&\St\gamma = \frac{\St\alpha-\St\beta}2,
                                     \\
   &\text{if~~}\St\alpha+\St\beta\le0: &&\St\sigma = \St\alpha+\St\beta+2\pi,
                                     &&\St\gamma = \frac{\St\alpha-\St\beta}2\pm\pi,
  \endaligned  
  \qquad\St\omega = \frac{\St\sigma}2,
}
making $\St\sigma$ strictly positive.
The~${\pm}$ choice in the redefinition of $\St\gamma$
does not affect \eqref{LambdaNew}.

\begin{figure}[t]
\centering%
\Bfig{.8\textwidth}{Fig4-Cum}{%
Example of given data, producing two kinds of long spirals.
}%
\end{figure}

All spirals, found in \RefFig{Fig1-Exa}, correspond to the correction $\St\alpha\ceq\St\alpha+2\pi$
(they intersect the left complement of the chord, $x<{-1},\,y=0$).
In \RefFig{Fig4-Cum} we see long spirals of both kinds, with
$\{\St\alpha+2\pi,\:\St\beta\}$ and  $\{\St\alpha,\:\St\beta+2\pi\}$.
The gap between two subfamilies is due to rejection of the discontinuous solution.
Spirals in \RefFig{Fig2-LogSp} are short.

In the same manner as in \cite{AKparab}, 
we find a conic arc, sharing invariants $\sigma$ and $Q$ with given data:
\Equa{EQmain}{%
       \sigma(p,q,w,j)=\St\sigma,\qquad Q(p,q,w,j)=\St Q.
}
The sought for spiral $\bar{z}(t)$ will be found as the M\"obius map
of the conic~$z(t)=x(t)+\iu y(t)$~\eqref{RatConic},
\Equa{Moebius}{%
     \bar{z}(t)=\Mob(z(t);z_0)\equiv\dfrac{z_0+z(t)}{1+z_0 z(t)},\quad\text{where}\quad
    z_0=\dfrac{r_0\Exp{\iu\lambda_0}-1}{r_0\Exp{\iu\lambda_0}+1}=
    \frac{r_0-r_0^{-1}+2\iu\sin\lambda_0}{r_0+r_0^{-1}+2\cos\lambda_0}
}
(see Proposition\,1 in \cite{AKparab}). The parameters of the map are
\Equa{Lambda12}{%
    \begin{array}{l}
     \lambda_0=\St\alpha{-}\alpha \equiv \beta{-}\St\beta\pmod{2\pi},\\[1.5ex]
     r_0=r_{01}=r_{02}{:}\quad
         r_{01}=\dfrac{a +\sin\alpha}{\St[1]g},\quad
         r_{02}=\dfrac{\St[2]g}{b - \sin\beta},
    \end{array}
}
with $\alpha,\beta,a,b$ defined from \eqref{ConicG2}.
Two versions of $\lambda_0$ become equivalent as soon as the first equation 
in~\eqref{EQmain} is satisfied. Satisfying the second one equates $r_{01}$ and $r_{02}$.

\subsection{Defining M\"obius invariant  $\sigma$ of a conic arc}

Now we define invariant $\sigma$ (lens' angular width) for a conic
with increasing curvature and control parameters $p,q,j$.
Boundary angles $\alpha,\beta$ \eqref{ConicG2},
being in the interval $(-\pi,\pi)$,
define exactly $\sigma=\alpha{+}\beta$ for a short arc of conic ($j=1$).
As established in \cite[Prop.\,3]{AKparab},
the locus of control points $(p,q)$, yielding $\sigma=\St\sigma$,
is the part of the hyperbola
\Equa{LocusEq}{%
   H(p,q;\St\sigma)=0,\quad q\ne0,\quad\text{where}\quad
   H(p,q;\St\sigma)=\sin\St\sigma(1-p^2+q^2)+2pq\cos\St\sigma,
}
lying in quadrants II, IV (i.\,e. $pq<0$,
the spirality being possible only within the unit circle).
The part of the hyperbola in quadrants I, III was useless,
when we worked with convex conic ($j={+1}$).
But it becomes useful as soon as discontinuous conics ($j={-1}$) are included into consideration:
\Equa{LocusFig}{%
   \parbox{360pt}{\Infigw{360pt}{Pic22}}%
}
\begin{Prop}
Conic~\eqref{RatConic} with $j=-1$ and the control point $(p,q)$ such that
\equa{
    pq>0, \quad H(p,q;\St\sigma)=0,\quad 0<\St\sigma<\pi,
}
provides the value of invariant $\sigma=\St\sigma$.
\end{Prop}
\begin{proof}[{\bf Proof}]
By construction, the whole locus $H(p,q;\St\sigma){=}0$ can supply conics with \ $\tan\sigma{=}\tan\St\sigma$, i.\,e.
$\sigma\in\{\St\sigma;\,\St\sigma{\pm}\pi;\,\ldots\}$; we have to reduce the choice
to the first possibility.

Let $(p,q)$, $p>0$, $q>0$, be a point on the locus~\eqref{LocusEq}.
Eqs.\,\eqref{LimCirc2}, \eqref{Wsign} require $p>1$, $w<0$.
This control point generates conic $z(t)$ with
boundary angles $\alpha,\beta$. The conic is discontinuous, and is located within
the unbounded lens of the width~$\sigma$ (shown shaded):
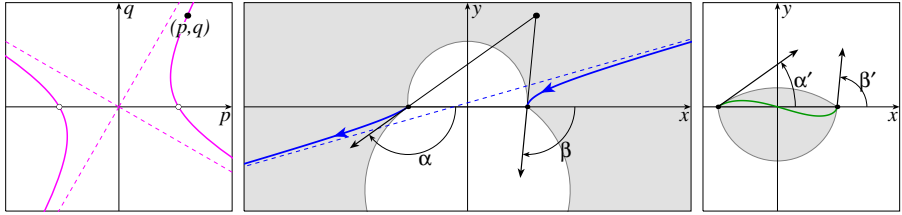
\equa{%
   \Infigw{.72\textwidth}{PicProof}
}
Consider lemniscate-like regular spiral, obtained from~$z(t)$
by the map $\Mob(z(t);\infty)=\frac1{z(t)}$~\eqref{Moebius}.
This map preserves $\sigma$, makes the spiral short (and the lens bounded), thus making $\sigma$ easy to calculate.

The map includes inversion with respect to the unit circle, followed by reflection about the $x$-axis.
Inversion converts $\alpha$ into $-\alpha\pm\pi$, reflection negates the result.
So, tangent angles of the curve-image become $\alpha'=\alpha\pm\pi$, $\beta'=\beta\pm\pi$, 
where ${+}$~or~${-}$ should be chosen simply to put each value into the range $(-\pi;\pi)$.
The invariant $\sigma$ is then exactly equal to $\alpha'+\beta'$.
From \eqref{ConicG2} with $p>1$, $w<0$, $j=-1$, we deduce
\equa{%
  \renewcommand{\tmp}{&\Longrightarrow&}
  \begin{array}{lclcl} 
    \cos\alpha<0,\quad\sin\alpha<0\tmp -\pi<\alpha<-\pi/2\tmp \alpha'=\alpha+\pi;\\
    \cos\beta<0, \quad\sin\beta<0 \tmp -\pi<\beta<-\pi/2 \tmp \beta'=\beta+\pi.
  \end{array} 
}
So, $\sigma=\alpha'+\beta'=\alpha+\beta+2\pi$, \ $0<\sigma<\pi$. , \ $\sigma=\St\sigma$.

Taking \LT{q}, we arrive to a conic with $w>0$, the control point in the opposite quadrant,
and the same conclusion for~$\sigma$.
\end{proof}

Let us parametrize locus \eqref{LocusEq} 
in terms of angular parameter $-\pi<\theta=2\nu<\pi$,
which will serve as the parameter of the family of solutions:
\Equa{pqnu}{%
  \aligned
   &p(\theta)=\frac{\sin\St\sigma}{\sin\theta}\equiv\frac{\sin\St\omega\cos\St\omega}{\sin\nu\cos\nu},
   \qquad \nu=\frac{\theta}{2},\\[1.5ex]
   &q(\theta)=-\frac{\cos\St\sigma-\cos\theta}{\sin\theta}\equiv
   \frac{\sin(\St\omega+\nu)\sin(\St\omega-\nu)}{\sin\nu\cos\nu};
 \endaligned\qquad
   \parbox[c]{135pt}{\Infigw{135pt}{Pic23}}
}
The case of infinite control point, $w=0$, omitted in Propositions 1,\,2,
can now be added to the family as $\theta=0$.
It is the infinite point in the direction of the asymptota $q=p\tan\St\omega$,
shown solid in~\eqref{pqnu},
in either the first $(\theta\to{+}0)$, or the third $(\theta\to{-}0)$ quadrant.
Both cases yield identical solutions.
In \eqref{pqnu} $0\le\theta<\St\sigma$ is accepted, assigning $\theta=0$ to the first quadrant,
in which $w<0$.
Fractions $\frac{w}{\abs{w}}$ in \eqref{ConicG2} take the limit value~$-1$ in this case.
 
The infinite point in the direction of the dashed asymptota
is achieved when $\theta\to\pm\pi$, 
and yields non-spiral elliptic arc~(\ref{NonSpirals}a).
The parameter range $\abs{\theta}\le\Theta_1$, 
where spiral solutions could be found, is limited by either  point $M$
with $p^2+q^2=1$ (\ref{LocusFig}a), 
or point $N$ with $p=1$ (\ref{LocusFig}b):
\Equa{Theta1}{%
   \Theta_1=\left\{\begin{array}{cll}
         \dfrac{\pi}{2}, &\text{if~~} 0<\St\sigma<\dfrac{\pi}{2},&      \text{case (\ref{LocusFig}a)}\\
         \pi-\St\sigma, &\text{if~~} \dfrac{\pi}{2}\le\St\sigma\le\pi,& \text{case (\ref{LocusFig}b)}
   \end{array}\right\}=\min\left(\dfrac{\pi}{2},\pi-\St\sigma\right).
}

The sides $h_{1,2}$ \eqref{h1h2} of the control polygon obey equalities
\equa{
    h_1^2=\dfrac{2-2\cos(\St\sigma+\theta)}{\sin^2\theta},
    \quad     
    h_2^2=\dfrac{2-2\cos(\St\sigma-\theta)}{\sin^2\theta},
}
or, taking into account signs~\eqref{Wsign} and~\eqref{pqnu},
\equa{
    h_1=j \frac{w}{\abs{w}}\cdot \frac{\sin(\St\omega+\nu)}{\sin\nu\cos\nu},\quad
    h_2=-\frac{w}{\abs{w}}\cdot \frac{\sin(\St\omega-\nu)}{\sin\nu\cos\nu}.%
}
Boundary G$^2$ data \eqref{ConicG2} now look like
\Equa{ConicG2new}{%
\hspace*{-5mm}%
  \begin{array}{lll}
     \cos\alpha =j\cos(\St\omega-\nu),&
    \sin\alpha = j\sin(\St\omega-\nu),&
    a=-\dfrac{\sin^2\nu\cos^2\nu\sin(\St\omega-\nu) }{w^2\sin^2(\St\omega+\nu)};
    \\[2ex]
    \cos\beta =j\cos(\St\omega+\nu),&
    \sin\beta =j\sin(\St\omega+\nu),&
    b=\HM{}\dfrac{\sin^2\nu\cos^2\nu\sin(\St\omega+\nu) }{w^2\sin^2(\St\omega-\nu)}.
  \end{array}
}

\subsection{Defining invariant $Q$ and weight $w$}

For every control point
the proper values of weight $w$ will be found by 
equating inversive invariants $\St Q=Q\;(=g_1g_2+\sin^2\omega)$ for given
and conic G$^2$~data.
Choosing control points on the locus~\eqref{LocusEq} assures $\St\omega=\omega$,
and reduces $\St Q=Q$ to $\St[1]{g}\St[2]{g}=g_1g_2$. From \eqref{ConicG2new}
\Equa{g1g2}{%
  \begin{split}
   g_1&{}=a+\sin\alpha=
         \sin(\St\omega-\nu)\left[j-\frac{\sin^2\nu\cos^2\nu}{w^2\sin^2(\St\omega+\nu)}\right],\\
   g_2&{}=b-\sin\beta=
         \sin(\St\omega+\nu)\left[-j+\frac{\sin^2\nu\cos^2\nu}{w^2\sin^2(\St\omega-\nu)}\right],
   \end{split}                 
   \quad
      g_1 g_2=\frac{\sin^4\theta-4jw^2 D_1 \sin^2\theta+4w^4 D_2^2}{8w^4D_2},
}
\Equa{DDDD}{%
   \begin{array}{llcll}
   \text{where~~}   
   &D_1=1-\cos\St\sigma\cos\theta&=&\sin^2(\St\omega+\nu)+\sin^2(\St\omega-\nu), &\hphantom{j}D_1>0;\\
   &D_2=\cos\St\sigma-\cos\theta &=&-2\sin(\St\omega+\nu)\sin(\St\omega-\nu), \quad &jD_2>0.\\[1.ex]
   \text{Define also~~}
   &D_3=D_2-2\St[1]g\St[2]g      &=&1-2\St Q-\cos\theta, &\hphantom{j}D_3 > 0;\\
   &D_0=D_1^2-D_2 D_3            &=&\sin^2\theta\sin^2\St\sigma+2\St[1]g\St[2]g D_2.
   \end{array}
}
Equation $\St[1]{g}\St[2]{g}=g_1g_2$ is biquadratic for the weight~$w$:
\equa{
  4w^4 D_2 D_3 -4jw^2{D_1}\sin^2\theta + \sin^4\theta=0,
}
The equivalent equation for $N=\dfrac{w^2}{\sin^2\theta}$, and its roots look like
\Equa{eqN}{%
     4N^2 D_2 D_3 -4jN{D_1} + 1=0,\qquad
     N_{1,2}=\frac12\,\frac{D_1\pm\sqrt{D_0}}{jD_2D_3};\qquad
     w=\pm\sin\theta\sqrt{N}.
}
The term $\sin\theta$, singled-out
in $w=\pm\sin\theta\sqrt{N}$,
will be cancelled out of fractions like
\equa{
   \abs{p_w}=\abs{pw}=\abs{\frac{\sin\St\sigma}{\sin\theta}\cdot\sqrt{N}\,\sin\theta},
}
thus eliminating singularity in treatment the case of infinite control point.
To get proper signs of $w,\, p_w,\, q_w$, we put the sign into the factor $n_w=\pm1$,
making $w<0$ in the first quadrant only~\eqref{Wsign}:
\Equa{PwQw}{%
   n_w=
       \sgn(\theta-\St\sigma),\quad 
   w=n_w\sin\theta\,\sqrt{N},\quad 
   p_w=n_w\sin\St\sigma\,\sqrt{N},\quad
   q_w=-n_w(\cos\St\sigma-\cos\theta)\sqrt{N}.
}    

Discriminant $D_0(\theta)$ is an even function of $\theta$, positive at $\theta=0$
$[D_0(0)=-4\St[1]g \St[2]g\sin^2\St\omega]$, and monotone decreasing in $[0;\pi]$.
So, the range $\abs\theta \le\Theta_0$, such that $D_0\ge0$,
is given by the equation $D_0(\Theta_0)=0$:
\Equa{Theta0}{%
    \cos\Theta_0=
    \frac{2\St[1]g \St[2]g\cos\St\sigma +\sin^2\St\sigma}%
    {\St[1]g \St[2]g-\sqrt{{\St[1]g}^2{\St[2]g}^2 + 2\St[1]g \St[2]g \sin^2\St\sigma\cos\St\sigma+\sin^4\St\sigma}}.
}

\subsection{Defining resulting spiral}

First, we rewrite spirality tests~\eqref{LimCirc} in terms of parameter~$\theta$.
E.\,g., test~\eqref{LimCirc2},
\equa{%
\hspace*{-5mm}\small
    j={-}1\;\wedge\; w>0\;\wedge\;K_1(p,q)\le0\quad\text{goes with}\quad
    -\St\sigma<\theta<0\quad (-\St\omega<\nu<0),\quad\text{i.\,e.}\quad
     \sin\theta<0,\quad \sin(\St\omega+\nu)>0, 
}
and can be transformed as
\equa{%
  \begin{split}
    w^2 K_1(p,q)={}&h_1^2w^2+p+1=4\sin^2(\St\omega+\nu)N+\frac{\sin\St\sigma}{\sin\theta}+1=\\
     ={}&\left[\frac{2\sin(\St\omega+\nu)}{\sin\theta}\right]
      \left[2N\sin(\St\omega+\nu)\sin\theta+\cos(\St\omega-\nu)\right]\le 0,
  \end{split}  
}
the first term being negative.
Similarly, the test $j={-1}\;\wedge\; w<0\;\wedge\; K_2(p,q)\le0$
is applied when $0\le\theta<\St\sigma$, and transforms to
\equa{
   h_2^2w^2-p+1\le 0\So 2N\sin(\St\omega-\nu)\sin\theta-\cos(\St\omega+\nu)\le 0.
}
Both, unified for $\abs\theta<\St\sigma$ ($\abs\nu<\St\omega$), take form~\eqref{SpiralTest2}.
Likewise, \eqref{LimCirc1} can be rewritten as~\eqref{SpiralTest1}.

Now let us express in terms of $\nu=\frac12\theta$ parameters $\lambda_0,r_0$ \eqref{Lambda12}.
Substitutions \eqref{ConicG2new} for $\lambda_0=\St\alpha-\alpha=\St\omega+\St\gamma-\alpha$
yield
\begin{subequations}\label{R0L0New}
\Equa{LambdaNew}{%
   \cos\lambda_0=j\cos(\St\gamma+\nu), \qquad   \sin\lambda_0=j\sin(\St\gamma+\nu).
}
To calculate $r_0$, combine its two versions~\eqref{Lambda12}: $r_0^2=r_{01}^{\ } r_{02}^{\ }$, i.\,e.
\Equa{r012New}{
  \begin{split}
   r_{01}^{\ }&{}=\frac{g_1}{\St[1]g}\Eqref{g1g2}
   \frac{\sin(\St\omega-\nu)}{\St[1]g}\cdot
     \left[j-\frac{1}{4N\sin^2(\St\omega+\nu)}\right];\\
%
   r_{02}^{\ }&{}=\frac{\St[2]g}{g_2}\Eqref{g1g2}
     \frac{\St[2]g}{\sin(\St\omega+\nu)}\cdot\left[\frac1{4N\sin^2(\St\omega-\nu)}-j\right]^{-1};\\
   r_0&{}=\sqrt{r_{01}^{\ } r_{02}^{\ }}=\sqrt{-\frac{\St[2]g}{\St[1]g}}\cdot
         \sqrt{-j\frac{\sin^3(\St\omega-\nu)}{\sin^3(\St\omega+\nu)}}\cdot
         \sqrt{\frac{4N\sin^2(\St\omega+\nu)-j}{4N\sin^2(\St\omega-\nu)-j}};
  \end{split}
}
\end{subequations}

To write explicitely the parametric equation~\eqref{Moebius} of the resulting spiral~$\bar{z}(t)$,
denote
\equa{
  X_0=r_0-r_0^{-1},\quad Y_0=2\sin\lambda_0,\quad W_0=r_0+r_0^{-1}+2\cos\lambda_0,\quad
  z_0=\dfrac{X_0+\iu Y_0}{W_0},\quad
  \text{and}\quad  W_1=r_0+r_0^{-1}-2\cos\lambda_0.
}
In the final expression functions $X(t),Y(t),W(t)$ from \eqref{RatConic} are abbreviated as $X,Y,W$:
\equa{
     \bar{z}(t)=\frac{(X_0^2 + Y_0^2+ W_0^2)XW + X_0W_0\left(X^2+ Y^2 + W^2\right)+%
     \iu\left[Y_0W_0\left(W^2-X^2-Y^2\right) + \left(W_0^2-X_0^2-Y_0^2\right) Y W\right]}
     {W_0^2 W^2 +2W_0 W\left[X_0 X-Y_0 Y\right]+\left(X_0^2+Y_0^2\right)\left(X^2+Y^2\right)}.
}
Because $X_0^2+Y_0^2=W_0W_1$, \ $W_0$ can be cancelled out:
\Equa{Final4}{%
  \begin{split}
   \bar{z}(t)&{}=\frac{%
     (W_0+ W_1)WX + X_0\left(X^2+ Y^2 + W^2\right)%
     +\iu\left[Y_0\left(W^2-X^2-Y^2\right) + \left(W_0-W_1\right) Y W\right]}%
     {W^2 +2W\left(X_0 X-Y_0 Y\right)+W_1\left(X^2+Y^2\right)}=\\
     &{}=\frac{%
     r_0\left[(X+W)^2 + Y^2\right]-r_0^{-1}\left[(X-W)^2 + Y^2\right]%
     +2\iu\left[\left(W^2-X^2-Y^2\right)\sin\lambda_0 + 2Y W\cos\lambda_0\right]}%
     {r_0\left[(X+W)^2 + Y^2\right]+r_0^{-1}\left[(X-W)^2 + Y^2\right]\,{}+{}\,
      2\left[\left(W^2-X^2-Y^2\right)\cos\lambda_0 - 2Y W\sin\lambda_0\right]}.     
  \end{split}    
}

\subsection{Step-by-step construction of the family of solutions}\label{sec:StepByStep}

We assume that G$^2$ Hermite data to be matched have been brought to the standard normalized
form~\eqref{GivenKK}.
Below the construction is described step-by-step.

\begin{figure}[t]
\centering%
\Bfig{.8\textwidth}{Fig5-Locus}{Locus \eqref{pqnu} for $\St\sigma=\Deg{30}$;
control points are taken with the step $\Delta\theta=\Deg{2}$;
bullets mark control points, yielding (for given $\St Q=-0.103$) spiral solutions.}%
\end{figure}

\begin{enumerate}
\item
Calculate  $\St[1]g,\St[2]g$, and invariant $\St Q$ \eqref{DefQ} from given data.
Continue if $\St Q<0$.
\item
If the curvature is decreasing ($\St a>\St b$), convert given data by assignments~\eqref{MakeIncr}.
Bring boundary angles $\St\alpha,\St\beta$ to the range $(-\pi;\pi]$, 
and define $\St\sigma$, $\St\omega$, $\St\gamma$, following \eqref{DefSigma}. 
The algorithm is not applicable for
data with wide lens, $\St\sigma>\pi$. 
Try to split the path, and apply the algorithm to each subpath (\cite[Sec.\,9]{AKparab}).
\item
The parameter range,
where spiral solutions could be found (\ref{Theta1},\ref{Theta0}), is
\Equa{Step3}{%
   \abs\theta\le\Theta,\quad \Theta=\min\left(\dfrac{\pi}{2},\pi-\St\sigma,\Theta_0\right).
}   
To scan the range, prepare an array of parameters, 
e.\,g., $\theta\in\{0,\,\pm\Delta\theta,\,\pm2\Delta\theta,\,\ldots\}$,
with some sufficiently small step $\Delta\theta$, avoiding $\theta=\pm\St\sigma$
(to avoid $q=0$).
Note that solution for $\theta=0$ definitely exists.

In \RefFig{Fig5-Locus} control points are chosen with $\Delta\theta=\Deg{2}$; 
sectors $\abs\theta\le\Theta$ are shown shaded.
\item
For every $\theta$ calculate $D_{1,2,3,0}$ \eqref{DDDD}.
Create one or two tuples $\{\theta,j,N\}$ with $N>0$, namely:
\equa{%
   \begin{array}{llll}
     \text{if~~}\abs\theta<\St\sigma,&\text{create}& \Keep{\theta}{j=-1}{N=N_2}; \\
     \text{otherwise}&\text{create}&  \Keep{\theta}{j=+1}{N=N_2}&\text{and~~}\Keep{\theta}{j=+1}{N=N_1};\\
     &\text{here}
     &\displaystyle N_1=\frac12\,\frac{j}{D_1+\sqrt{D_0}},
     &\displaystyle N_2=\frac12\,\frac{D_1+\sqrt{D_0}}{jD_2D_3}.
   \end{array}     
}
%
\item
For every tuple perform spirality test:
\begin{subequations}\label{SpiralTest}
\begin{align}
  &\makebox[0mm][r]{if $j=+1$:}\quad
  \left\{\begin{array}{l}
    [2N\sin(\St\omega+\nu)\sin\theta-\cos(\St\omega-\nu)]
    [2N\sin(\St\omega-\nu)\sin\theta+\cos(\St\omega+\nu)] \ge 0,\\
    2N\sin^2\theta\ge 1\quad\left(\text{i.\,e.~~} w\ge\frac{1}{\sqrt{2}}\right);
  \end{array}\right.    \label{SpiralTest1}\\
  &\makebox[0mm][r]{if $j=-1$:}\quad
      2N\sin(\St\omega-\abs\nu)\sin\abs\theta-\cos(\St\omega+\abs\nu)\le 0.\label{SpiralTest2}
\end{align}     
\end{subequations}    
Reject the tuple, if the test fails. Attach $w,\,p_w,\,q_w$ \eqref{PwQw} to the retained tuples.
\item
For every 6-tuple $\{\theta,\,j,\,N,\,w,\,p_w,\,q_w\}$
there exists a spiral conic arc $z(t)=x(t)+\iu y(t)$~\eqref{RatConic}.
Define parameters $r_0,\,\lambda_0$ \eqref{R0L0New}.
The resulting curve $\bar{z}(t)=\bar{x}(t)+\iu\bar{y}(t)$ is given by \eqref{Final4}.
\end{enumerate}

\noindent%
Returning to decreasing curvature, if it was the case in Step\,2, is done by negating $\bar{y}(t)$.

\section{Reducing rational 4th degree interpolant to 3rd degree}\label{sec:Cubic}

\enlargethispage{1\baselineskip}
Map \eqref{Moebius} can be decomposed into elementary transforms, namely
\equa{%
    \Mob(z,z_0)=
    \frac{1-z_0^{-2}}{z+z_0^{-1}}+z_0^{-1}.
}
The last term is responsible for translation,
the numerator performs scaling~+~rotation,
and the denominator includes inversion~+~reflection.
Only inversion affects the degree of the curve-image. 
The center of inversion is the point $z_1=-z_0^{-1}$.

If the center of inversion lyes on the conic, the resulting 4th degree curve is reducible to 3rd degree.
To see it, let us take the center~$z_1$ on the original curve~\eqref{RatConic}:
\equa{%
   z_1=z(T)=\frac{X_1}{W_1}+\iu\frac{Y_1}{W_1},\quad
   X_1=X(T),\quad Y_1=Y(T),\quad W_1=W(T).
}
Inversion+reflection look like
\equa{
     \dfrac{1}{z(t)-z_1}=\dfrac{W_1W(t)}%
     {[\,\underbrace{W_1X(t)-X_1W(t)}_{{}=(t-T)A(t)}\,]+
     \iu[\,\underbrace{W_1Y(t)-Y_1W(t)}_{{}=(t-T)B(t)}\,]}
     =\dfrac{W_1W(t)\left[A(t)-\iu B(t)\right]}%
     {(t-T)\left[A^2(t)+B^2(t)\right]}\,.
}
Polynomials $A(t)$ and $B(t)$ being linear, the resulting curve is 3rd degree rational.\medskip

The center of inversion of map \eqref{Moebius} is
\Equa{X1Y1W1}{%
    z_1=-z_0^{-1}=\dfrac{1+r_0\Exp{\iu\lambda_0}}{1-r_0\Exp{\iu\lambda_0}}=\frac{X_1+iY_1}{W_1},
    \quad\text{where}\quad
    \begin{array}{l}
       X_1=r_0^{-1}-r_0,\quad Y_1=2\sin\lambda_0,\\[1ex]
       W_1=r_0^{-1}+r_0-2\cos\lambda_0.
    \end{array}   
}
Condition that the center \eqref{X1Y1W1} belongs to the conic~\eqref{Implicit},
\equa{
     q^2{X_1^2} - 2pq {X_1Y_1}+\left(p^2+\frac{j}{w^2}-1\right){Y_1^2}+2q{Y_1}{W_1}-q^2W_1^2=0,
}
takes form
\equa{
  q q'(r_0+r_0^{-1})+pq(r_0-r_0^{-1})\sin\lambda_0-q^2-{q'}^2
     +\sin^2\lambda_0\left(p+\frac{j}{N\sin^2\theta}\right)=0,\quad
    \text{where}\quad q'=q\cos\lambda_0+\sin\lambda_0,
}
and $w^2$ is replaced according to \eqref{eqN}.
The result is linear in $N^{-1}$, and remains such after substituting
$r_{01}\pm r_{02}^{-1}$ for        $r_{0}\pm r_{0}^{-1}$ \eqref{r012New}.
Further substitutions, \eqref{pqnu}, \eqref{LambdaNew}, simplify to
\equa{
   2jN=\frac{1}{\cos\St\sigma-\cos\theta}\cdot\frac{f_1(\theta)}{f_2(\theta)},
}
where $f_{1,2}$ are linear functions of $\cos\theta,\,\sin\theta$.
Conversion to polynomials of $v=\tan\frac\theta{2}$ looks like
\Equa{Nv}{%
   4jN=\frac{1+v^2}{v^2\cos^2\St\omega-\sin^2\St\omega}\cdot
       \frac{A_1 v^2 + B_1 v + C_1}{A_2 v^2 + C_2},
}
where
\equa{
  \aligned
   &&B_1&=\St[1]g\St[2]g\sin(\St\alpha{-}\St\beta)+(\St[1]g\sin\St\beta+\St[2]g\sin\St\alpha)\sin(\St\alpha{+}\St\beta);\\[1ex]
   &&A_1-A_2
          &=2\St[1]g\St[2]g\sin\St\alpha\sin\St\beta;\\[1ex]
   &&A_1+A_2&=2\cos^2\St\omega(\St[1]g\St[2]g+\St[1]g\sin\St\beta-\St[2]g\sin\St\alpha);\\[1ex]
   &&C_1-C_2&=2\sin^2\St\omega(\St[1]g\St[2]g+\St[1]g\sin\St\beta-\St[2]g\sin\St\alpha);\\[1ex]
   &&C_1+C_2&=-(A_1+A_2).
  \endaligned
}
\Skip{
   \begin{array}{ll}
   A_1=\St[1]g\St[2]g\cos^2\St\gamma+\cos^2\St\omega(\St[1]g\sin\St\beta-\St[2]g\sin\St\alpha),& 
   A_1-A_2=2\St[1]g\St[2]g(\cos^2\St\gamma-\cos^2\St\omega)
          =2\St[1]g\St[2]g\sin\St\alpha\sin\St\beta;\\[1.5ex]
   B_1=\St[1]g\St[2]g\sin2\St\gamma+\sin2\St\omega(\St[1]g\sin\St\beta+\St[2]g\sin\St\alpha);\\[1.5ex]
   C_1=\St[1]g\St[2]g\sin^2\St\gamma+\sin^2\St\omega(\St[1]g\sin\St\beta-\St[2]g\sin\St\alpha),&
   C_1-C_2=2\sin^2\St\omega(\St[1]g\St[2]g+\St[1]g\sin\St\beta-\St[2]g\sin\St\alpha).
   \end{array}
}
\Skip{%
\equa{%
\begin{split}
    f_1(\theta)={}&
    2\St[1]g\St[2]g[1-\cos(\St\alpha - \St\beta - \theta)]+
      \St[1]g[\sin(\St\alpha-\theta)-\sin(\St\sigma+\St\beta -\theta)+2\sin \St\beta]-\\
    -&\St[2]g[\sin(\St\beta +\theta)-\sin(\St\sigma+\St\alpha-\theta)+2\sin\St\alpha],\\
   f_2(\theta)={}&
   2\St[1]g\St[2]g(2\cos\St\sigma-\cos2\St\gamma-\cos\theta)-
     \St[1]g[\sin(\St\beta +\theta)+\sin(\St\beta -\theta)-\sin(\St\beta +\sigma)+\sin\St\alpha]+\\
   &+\St[2]g[\sin(\St\alpha+\theta)+\sin(\St\alpha-\theta)-\sin(\St\alpha+\St\sigma)+\sin\St\beta].
\end{split} 
}
}   
It remains to substitute $N$ into \eqref{eqN}, rewritten below in terms of~$v$:
\equa{
  \left(\frac{4jN}{1+v^2}\right)^2(v^2\cos^2\St\omega-\sin^2\St\omega)
  [v^2\cos^2\St\omega-\sin^2\St\omega-g_1g_2(1+v^2)]
  -\frac{8jN}{1+v^2}(v^2\cos^2\St\omega+\sin^2\St\omega)+1=0
}
We obtain the 6th degree algebraic equation
\Equa{EQ6}{%
   \aligned
   \cos^2\St\omega&[(A_1{-}A_2)v^2+B_1v+C_1{-}C_2]^2 v^2-{}\\
   -\sin^2\St\omega&[(A_1{+}A_2)v^2+B_1v+C_1{+}C_2]^2-
    \St[1]g\St[2]g (A_1v^2+B_1v+C_1)^2(1+v^2)=0.  
   \endaligned 
}
\Skip{%
We obtain the 6th degree algebraic equation $v_6v^6+v_5v^5+\ldots v_0=0$:
\equa{
  \begin{array}{l}
     v_6=(A_1-A_2)^2\cos^2\omega-g_1g_2A_1^2,\\[1.5ex]
     v_5=4B_1(A_1-A_2)\cos^2\omega-4g_1g_2A_1B_1,\\[1.5ex]
     v_4=(A_1+A_2)^2\cos^2\omega-(A_1-A_2)^2-g_1g_2(A_1^2+2A_1C_1+4B_1^2)+
        2[2B_1^2+(C_1-C_2)(A_1-A_2)]\cos^2\omega,\\[1.5ex]
     v_3=4B_1[2(A_1+C_1)\cos^2\omega-g_1g_2(A_1+C_1)-(A_1+A_2)],\\[1.5ex]
     v_2=(C_1-C_2)^2\cos^2\omega-g_1g_2(C_1^2+2A_1C_1+4B_1^2)-
        2[2B_1^2+(C_1+C_2)(A_1+A_2)]\sin^2\omega,\\[1.5ex]
     v_1=-4B_1(C_1+C_2)\sin^2\omega-4g_1g_2C_1B_1,\\[1.5ex]
     v_0=-(C_1+C_2)^2\sin^2\omega-g_1g_2C_1^2.
  \end{array} 
}
}

\begin{figure}[t]
\centering%
\Bfig{.92\textwidth}{Fig6-Cubics}{%
Existence of rational cubics, inversions of conics.
Dots mark regions in curvature space $(a,b)$
for fixed pairs $\alpha,\,\beta$,
where Eq.\eqref{EQ6} yields rational cubic.
Heavy dots (black circles) mark spiral solutions.
}%
\end{figure}

Finding roots of polynomials does not pose numerical problems.
For each real root~$v$ define \hbox{$\nu=\arctan v$}, $\theta=2\nu$, $j=\sgn(\abs{\theta}-\St\sigma)$,
and $N$ from~\eqref{Nv}. Keep only roots, yielding $N>0$.
If the solution passes spirality test \eqref{SpiralTest},
define $r_0$, $\lambda_0$.
As the center~\eqref{X1Y1W1} is the point $x(T),y(T)$ on the conic~\eqref{RatConic}, 
$T$~can be found from the system of equations
\equa{%
         W_1{\cdot}X(T)=X_1{\cdot}W(T) \quad\wedge\quad W_1{\cdot}Y(T)=Y_1{\cdot}W(T),
}
considered as linear system in $T$ and~$T^2$:
\equa{
  \begin{split} 
   T
    &{}=\frac{q_w(X_1+W_1)+(j-p_w-w)Y_1}{q_w[X_1+W_1-j(X_1-W_1)]+[j(p_w-w+2)-(p_w+w)]Y_1}=\\
    &{}=\frac{(p_w+w-j)\sin\lambda_0+q_w(\cos\lambda_0-r_0^{-1})}%
          {[p_w+w-j(p_w-w+2)]\sin\lambda_0+q_w(1+j)\cos\lambda_0-q_w(r_0^{-1}+jr_0)}.
  \end{split}  
}
Curve \eqref{Final4} becomes cubic after cancellation of $(t-T)$ from its numerator and denominator.
\enlargethispage{\baselineskip}
To express it explicitely, denote\vspace{-\baselineskip}
\equa{%
  \Ber{n}{x}{a_0,\ldots,a_n}=\sum\limits_{i=0}^{n}a_i(1-x)^{n-i}x^i,\quad\text{and}\quad
  \bar{z}(t)=\dfrac{X_3(t)+\iu Y_3(t)}{W_3(t)}:
}  
\equa{%
\aligned
     &\aligned 
      X_3(t)=\ber{3}{t}\Big(%
     &       \Ber{3}{T}{0,\, h_{1_w}^2, \, 2j(p_w+w),\,1},\\
     &       \Ber{1}{T}{1,\,-2(p_w-w)}\cdot\Ber2{T}{ h_{1_w}^2, \,2j(p_w+w),\,1},\\
     &       \Ber1{T}{2j(p_w+w),\,\,1\,}\cdot\Ber2{T}{1,\,-2(p_w-w),\, h_{2_w}^2},\\
     &       \Ber3{T}{1,\,-2(p_w-w),\, h_{2_w}^2, \, 0}\Big);
     \endaligned
     \\[1ex]
     &Y_3(t)=2q_w\cdot\Ber{3}{t}{0,\:%
     \Ber{3}{T}{0, \, j- h_{1_w}^2,\,  -2jp_w, \, 0},\:
     \Ber{3}{T}{0, \, -2jp_w,\,jh_{2_w}^2-1,\,0},\:0};
     \\[1ex]
     &\aligned
        W_3(t)=(t-T)\cdot\ber{2}{t}\Big(%
        &  \Ber{2}{T}{ h_{1_w}^2,\,2j(p_w+w),\,1},
          2\Ber{2}{T}{j(p_w+w),\, 1-j(p_w^2+q_w^2-w^2),\,-(p_w-w)},\\
        &  \Ber{2}{T}{1,\,-2(p_w-w),\,h_{2_w}^2}\Big).
     \endaligned   
\endaligned
}

\Skip{
\equa{
   \aligned       
     X_3(t)&{}=(1-t)^3\cdot T \left[(1-T)^2 h_{1_w}^2+2T(1-T)j(p_w+w)+T^2 \right] +\\
           &{}+t(1-t)^2\cdot\left[1-T+2T(w+p_w)\right]\left[(1-T)^2 h_{1_w}^2+2T(1-T)j(p_w+w)+T^2 \right]+\\
           &{}+t^2(1-t)\cdot\left[T+2j(1-T)(w+p_w)\right]\left[(1-T)^2-2T(1-T)j(p_w-w)+T^2 h_{2_w}^2 \right]+\\
           &{}+t^3\cdot(1-T)\left[(1-T)^2+2T(1-T)(p_w-w)+T^2 h_{2_w}^2 \right];\\[1.5ex]
     Y_3(t)&{}=2t(1-t)T(1-T)q_w\cdot
           \left\{(1-t)\left[(1-T)(j-h_{1_w}^2)-2Tjp_w\right]
             +t\left[T(jh_{2_w}^2-1)-2(1-T)jp_w \right]\right\};\\[1.5ex]
     W_3(t)&{}=-(1-t)^3\cdot T \left[(1-T)^2 h_{1_w}^2+2T(1-T)j(p_w+w)+T^2 \right] +\\
           &{}+t(1-t)^2\cdot\left[(1-T)^3 h_{1_w}^2+T^2(1-T)(2j(p_w^2+q_w^2-w^2)-1)-2T^3(w-p_w) \right]+\\
           &{}+t^2(1-t)\cdot\left[2(1-T)^3j(p_w+w)-T(1-T)^2(2j(p_w^2+q_w^2-w^2)-1)-T^3 h_{2_w}^2) \right]+\\
           &{}+t^3\cdot(1-T)\left[(1-T)^2-2T(1-T)(p_w-w)+T^2 h_{2_w}^2 \right];
   \endaligned
}  
}

\RefFig{Fig6-Cubics}
illustrates existence of cubic solutions
as regions in the curvature space $(a,b)$.
Every plot is prepared for a fixed pair $\{\alpha,\beta\}$
with either $\sigma=\Deg{40}$ or $\sigma=\Deg{120}$.
The region $(a,b)$, allowing existence of general spiral,
is bounded by the branch $a<b$ of the hyperbola $Q(a,b)=0$.
Dots mark points $(a,b)$,
where solutions of Eq.~\eqref{EQ6} with $N>0$ exist.
Heavy dots distinguish spiral solutions.
Swapping $\alpha$ and $\beta$ would result in symmetric picture.
Plots in the right panel with $\alpha<0$ and $\beta<0$ assume either $\alpha\ceq\alpha+2\pi$, or $\beta\ceq\beta+2\pi$
\eqref{DefSigma}.
For $\sigma=\Deg{120}$ no solutions were found with $\{\alpha,\beta\}=\{\Deg{60},\Deg{60}\}$,
$\{\Deg{80},\Deg{40}\}$,
$\{\Deg{100},\Deg{20}\}$,
$\{\Deg{120},\Deg{0}\}$.
\smallskip

\begin{figure}[t]
\centering%
\Bfig{.8\textwidth}{Fig7-Dietz}{(a) Comparison of this algorithm with \cite{DietzRational},
and the selected solution (the lower left black circle); 
(b,c,d)~details of the selected solution:
(b)~the conic, its control polygon $APB$, 
bounding circle~\eqref{LimCirc2}, and the circle of inversion;
(c)~solution itself; (d)~its curvature plot.
}%
\end{figure}
 
Another approach to construct rational cubics was proposed in \cite{DietzRational}.
To compare results, we partially
reproduce Figure\,7 from \cite{DietzRational} as \RefFig[(a)]{Fig7-Dietz} herein.
First, note
that notation $\phi_0,K_0,\phi_1,K_1$ for boundary conditions corresponds to our notation
as $\St\alpha=-\phi_0$, $\St\beta=\phi_1$, $\St a=\frac12 K_0$, $\St b=\frac12 K_1$:
the difference in normalized curvatures is due to different chord lengthes:
our curve starts from $(-1,0)$, not from $(0,0)$, and has the chord length~2.
Both scales, $(a,b)$ and original $(K_0,K_1)$, are shown in \RefFig[(a)]{Fig7-Dietz}.
Squares are simply copied from the original figure,
where they mark curvatures, for which a rational cubic spiral
was found in \cite{DietzRational}.

Comparison with three other examples, Figures 8, 9, 10 in \cite{DietzRational},
shows regions with solutions, found in \cite{DietzRational}, 
and not found by this algorithms.
But the general feature is that the inversion of conics finds more convex cubic spirals,
and, additionally, non-convex ones.

\begin{figure}[t]
\centering%
\Bfig{.92\textwidth}{Fig8-Cornu}{Three examples with symmetric boundary conditions,
$\St\alpha=\St\beta=\Deg{45}$, $\St a=-\St b$.
Curvatures $\pm2.2$ match the arc of Cornu spiral.
}%
\end{figure}

One of solutions in \RefFig[(a)]{Fig7-Dietz}
is selected for detailed illustration, and as the numerical example.
The boundary angles and curvatures are
$\St\alpha=-0.1\;[\approx\Deg{-5.7}]$, $\St\beta=1.5\;[\approx\Deg{85.9}]$,
$\St{a}=0.0$,
$\St{b}=8.26$.
Equation~\eqref{EQ6} transforms to
\equa{%
    v^6 - 1.34748v^5 - 0.942759v^4 + 1.02859v^3 - 0.042459v^2  + 0.318889 v + 0.056006 = 0.
}
One of its roots, $v\approx-0.1582$, yields cubic rational spiral with 
$\theta\approx{-0.3137}\,[{\approx}\,\Deg{-17.97}]$,
$j=-1$, $N\approx1.861>0$
($p_w\approx-1.3445$, $q_w\approx-1.0659$,
 $w\approx 0.4210$, $\lambda_0\approx2.185$, $r_0\approx11.38$, $T\approx-0.0612$).
In \RefFig[(b)]{Fig7-Dietz} the conic (discontinuos hyperbola) is shown by dotted and solid lines,
solid for the parameter range $t\in[0;1]$.
The curve and its curvature plot are shown in Figures \Figref{Fig7-Dietz}(c) and~\Figref{Fig7-Dietz}(d).

\section{Conclusions}

The general algorithm, involving all possible conics, turned out to be quite simple and
straightforward;
solving biquadratic equation seems to be its most complicated part.

Testing the algorithm with different boundary conditions,
borrowed from known spirals, such as
logarithmic spiral (\RefFig{Fig2-LogSp}), 
Cornu spiral (the second example in \RefFig{Fig8-Cornu}),
other spiral curves, including conic arcs themselves,
has shown that the whole family did not deviate much from the
parent spiral. 
Visual comparison is often sufficient to select the best interpolant
to a given curve.

The initial idea to provide
wide variety
of shapes
is not put into big effect:
in most cases the whole family of interpolating spirals
occupies rather narrow region within the bilens,
which is the exact bound for all possible spiral interpolants (see \cite{Bilens}).
Bilenses are shown shaded in \RefFig{Fig8-Cornu}.
Nevetherless, there remains a big freedom to modify the path (and curvature profile)
by choosing intermediate curvature element
at some point~$M$ within the bilens.
With two families of interpolants, one on the chord $AM$, the other on $MB$,
the user can try to fulfil additional requirements,
like, e.\,g., G$^3$ continuity at the join point.

Analysis in Section\,4{.}2 shows
that the solution with infinite control point
covers the most wide range of boundary G$^1$ data, namely, $\abs{\St\sigma}\le\pi$.
According to \cite{AKhyperb}, the solution exists for any boundary curvatures, compatible with spirality ($\St Q<0$).
This solution could be recommended as the universal one for
the cases, where extra freedom is not needed.


\end{document}